\documentclass[12pt]{article}

\usepackage{mathtools,amsthm,amssymb}
\usepackage{hyperref}
\usepackage{tikz-cd}

\usepackage[margin=1in]{geometry}
\usepackage{microtype}

\usepackage{titlesec}

\titleformat{\section}{\center\normalsize\bfseries}{\thesection. }{0ex}{}

\titleformat{\subsection}{\normalsize\bfseries}{\hspace{3em}\thesubsection. }{0pt}{}

\titleformat{\subsubsection}[runin]{\normalsize\bfseries}{\thesubsection}{0.5em}{}[.]

\usepackage{enumitem}
\setlist{topsep=0.5ex,itemsep=0ex,parsep=0ex,leftmargin=3em,listparindent=\parindent}

\newtheoremstyle{thmstyle}{}{}{\itshape}{}{\bfseries}{. }{0pt}{\thmnumber{#2.\ }\thmname{#1}\thmnote{\textnormal{ (#3)}}}

\newtheoremstyle{mainthmstyle}{}{}{\itshape}{}{\bfseries}{. }{0pt}{\thmname{#1 }\thmnumber{#2}\thmnote{\textnormal{ (#3)}}}

\newtheoremstyle{defnstyle}{}{}{}{}{\bfseries}{.}{ }{\thmnumber{#2.\ }\thmname{#1}\thmnote{\textnormal{ (#3)}}}

\newtheoremstyle{thingstyle}{}{}{}{}{\bfseries}{. }{0pt}{}

\theoremstyle{thmstyle}
\newtheorem{thm}{Theorem}
\newtheorem{lem}[thm]{Lemma}
\newtheorem{prop}[thm]{Proposition}
\newtheorem{cor}[thm]{Corollary}

\theoremstyle{mainthmstyle}
\newtheorem{mainthm}{Theorem}

\theoremstyle{defnstyle}
\newtheorem{defn}[thm]{Definition}
\newtheorem{question}[thm]{Question}
\newtheorem{rmk}[thm]{Remark}

\theoremstyle{thingstyle}
\newtheorem{thing}[thm]{}

\counterwithin{thm}{subsection}\counterwithin{equation}{thm}

\DeclareMathOperator{\Aut}{Aut}

\DeclareMathOperator{\charac}{char}

\DeclareMathOperator{\Frob}{Frob}
\DeclareMathOperator{\Gal}{Gal}
\DeclareMathOperator{\GL}{GL}

\DeclareMathOperator{\inn}{inn}

\DeclareMathOperator{\Img}{Img}
\DeclareMathOperator{\Ker}{Ker}

\DeclareMathOperator{\SL}{SL}
\DeclareMathOperator{\Out}{Out}

\DeclareMathOperator{\rank}{rank}
\DeclareMathOperator{\Spec}{Spec}

\DeclareMathOperator{\W}{W}

\newcommand{\mc}{\mathcal}

\newcommand{\mr}{\mathrm}
\newcommand{\op}{\operatorname}
\newcommand{\tit}{\textit}

\newcommand{\ol}{\overline}
\newcommand{\wt}{\widetilde}

\newcommand{\A}{\mathbf{A}}

\newcommand{\G}{\mathbf{G}}
\newcommand{\cO}{\mathcal{O}}

\newcommand{\Q}{\mathbf{Q}}
\newcommand{\R}{\mathbf{R}}
\newcommand{\Z}{\mathbf{Z}}

\newcommand{\cF}{\mathcal{F}}
\newcommand{\cL}{\mathcal{L}}
\newcommand{\cR}{\mathcal{R}}
\newcommand{\cS}{\mathcal{S}}

\newcommand{\ff}{\mathfrak{f}}
\newcommand{\fT}{\mathfrak{T}}

\newcommand{\Ks}{{K^\mathrm{s}}}
\newcommand{\ssimp}{\mathrm{ss}}
\newcommand{\Zar}{\mathrm{Zar}}
\newcommand{\Xup}{\mathrm{X}^*}

\newcommand{\ceq}{\coloneqq}
\newcommand{\cln}{\colon}

\newcommand{\into}{\hookrightarrow}
\newcommand{\onto}{\mathrel{\mathrlap{\rightarrow}\mkern-2.25mu\rightarrow}}

\newcommand{\iso}{\overset{\sim}{\to}}

\let\phi\varphi
\let\setminus\smallsetminus

\newcommand{\Bibkeyhack}[3]{}

\title{
\large{\textbf{Transport of Zariski density in compatible collections of \texorpdfstring{\boldmath$G$}{G}-representations}}
}
\author{
\normalsize{\textbf{Jake Huryn and Yifei Zhang}}
}
\date{}

\begin{document}
\maketitle

\begin{abstract}
Let $X$ be a connected normal scheme of finite type over $\mathbf{Z}$, let $G$ be a connected reductive group over $\mathbf{Q}$, and let $\{\rho_\ell\colon\pi_1(X[1/\ell])\to G(\mathbf{Q}_\ell)\}_\ell$ be a Frobenius-compatible collection of continuous homomorphisms indexed by the primes.
Assume $\mathrm{Img}(\rho_\ell)$ is Zariski-dense in $G_{\mathbf{Q}_\ell}$ for all $\ell$ in a nonempty finite set $\cR$.
We prove that, under certain hypotheses on $\cR$ (depending only on $G$), $\Img(\rho_\ell)$ is Zariski-dense in $G_{\Q_\ell}$ for all $\ell$ in a set of Dirichlet density $1$.
As an application, we combine this result with a version of Hilbert's irreducibility theorem and recent work of Klevdal--Patrikis to obtain new information about the ``canonical'' local systems attached to Shimura varieties not of Abelian type.
\end{abstract}

\section{Introduction}

See \textsection\ref{subsec:notation} for a summary of the notation used throughout.

\subsection{Motivic background}
\label{subsec:motivic-bg}

\begin{thing}
\label{thing:tate-stuff}
In this item, let us assume all the standard conjectures on pure motives.
Let $X$ be a connected normal scheme of finite type over $\Z$.
Consider a collection $\{\rho_\ell\cln\pi_1(X_{\Z[1/\ell]})\to\GL_n(\Q_\ell)\}_\ell$ of (continuous) representations of the \'etale fundamental group of $X$.
Suppose $\{\rho_\ell\}_\ell$ is of geometric origin, e.g.\ $\rho_\ell$ arises as the monodromy representation of the $\ell$-adic local system $R^if_*\Q_\ell$ for a fixed smooth proper family $f$ over $X$.
Tate's conjecture then has the following consequence:
the \tit{algebraic monodromy groups} $M_\ell\ceq\ol{\Img(\rho_\ell)}{}^\Zar$ are \tit{independent of $\ell$} in the sense that $M_\ell\cong G_{\Q_\ell}$ for some algebraic group $G$ over $\Q$ \cite[Proposition 7.3.2.1]{andre}.
This $G$ is the ``motivic Galois group'' of the family $f$,\footnote{
More precisely, of the pure numerical motive over the function field of $X$ defined by $f$.
}
and a theorem of Jannsen \cite{jannsen} predicts that $G$ is a (possibly non-connected) reductive group.

On the other hand, by the Riemann hypothesis for varieties over finite fields, $\{\rho_\ell\}_\ell$ is \tit{compatible} in the sense that for any closed point $s$ of $X$, the characteristic polynomial of $\rho_\ell(\Frob_s)$, an element of $\Q_\ell[t]$, lives in $\Q[t]$ and is independent of $\ell$ when $\ell\neq\charac(\kappa_s)$.

Assuming Tate's conjecture, the Tannakian formalism allows one to upgrade $\{\rho_\ell\}_\ell$ to a collection $\{\wt\rho_\ell\cln\pi_1(X_{\Z[1/\ell]})\to G(\Q_\ell)\}_\ell$ of \tit{$G$-representations} all having Zariski-dense image which is compatible in the sense that $\{\xi\circ\wt\rho_\ell\}_\ell$ is compatible for any $\Q$-representation $\xi\cln G\to\GL_r$.
If $G$ is connected, this is equivalent to the following apparently stronger form of compatibility:
For each closed point $s$ of $X$, there exists an element $g\in G(\ol\Q)$ whose conjugacy class is defined over $\Q$ (i.e.\ $\Gal_\Q$-stable) such that the semisimple part of $\rho_\ell(\Frob_s)$ is conjugate in $G(\ol{\Q_\ell})$ to $g$.\footnote{
We have to fix some embeddings $\ol\Q\into\ol{\Q_\ell}$, but nothing will depend on them.
Note that, in general, a conjugacy class defined over $\Q$ will not contain an element defined over $\Q$, i.e.\ we cannot take $g\in G(\Q)$.

The equivalence stated here is an immediate consequence of the facts that the ring of class functions on a connected reductive group
is generated by the characters of irreducible representations,
separates semisimple conjugacy classes, and
detects the field of definition of a semisimple conjugacy class.

An alternative statement of compatibility uses the variety $G/\!\!/G$ of semisimple conjugacy classes, defined as the GIT quotient $\Spec(\cO(G)^G)$ with respect to the conjugation action.
For an algebraically closed field $F$, the map $[-]\cln G\to G/\!\!/G$ induces a bijection between $(G/\!\!/G)(F)$ and the set of semisimple conjugacy classes in $G(F)$.
The compatibility of $\{\wt\rho_\ell\}_\ell$ means that the elements $[\wt\rho_\ell(\Frob_s)]\in(G/\!\!/G)(\Q_\ell)$, for $\ell\neq\charac(\kappa_s)$, arise from a common element of $(G/\!\!/G)(\Q)$.
}
\end{thing}

\noindent
In an attempt to verify some consequences of these conjectures about pure motives, this paper considers the following question:

\begin{question}
\label{question:1}
Given an ``abstract'' compatible collection $\{\rho_\ell\cln\Gamma\to G(\Q_\ell)\}_\ell$ (a notion formalized in Definition \ref{defn:compatible} below), to what extent are its algebraic monodromy groups $M_\ell\ceq\ol{\Img(\rho_\ell)}{}^\Zar$ independent of $\ell$?
\end{question}

\noindent
Work of Serre implies that the rank and component group of $M_\ell$ are independent of $\ell$ \cite[\textsection2]{serre-letter}, \cite[2.2.3]{serre-resume}.
By delicate group theoretic arguments, Larsen--Pink obtained various other partial results in \cite{larsen-pink}.
For example, they prove that the dimension and Weyl group of $M_\ell$ depend, for $\ell$ from a set of Dirichlet density $1$, only on the Frobenius conjugacy class of $\ell$ in a fixed number field.
When $\Gamma$ is the absolute Galois group of a number field, more is known, e.g.\ \cite[Chapter 3]{serre-abelian-l-adic}, \cite{hui}, as well as when $\Gamma$ is the \'etale fundamental group of a smooth variety over a finite field, e.g.\ \cite{chin}, \cite{fancy-drinfeld}.

We do not approach Question \ref{question:1} in complete generality, but rather focus on the following special case:

\begin{question}
\label{question:main}
If $M_{\ell_0}=G_{\Q_{\ell_0}}$ for one $\ell_0$, can we conclude that $M_\ell=G_{\Q_\ell}$ for all $\ell$ in a set of Dirichlet density $1$?
(From \cite[Counterexample 10.4]{larsen-pink} we see that, in this abstract setting, the answer is ``\tit{no}'' if we ask that $M_\ell=G_{\Q_\ell}$ for all but finitely many $\ell$.)
\end{question}

\subsection{First main result}

Our first result gives a partial positive answer to Question \ref{question:main}, the caveat being that in many situations, we must replace $\ell_0$ by a finite set of primes.

\begin{defn}
Let $G$ be a connected reductive group over $\Q$, and let $E|\Q$ be the minimal extension such that $\Gal_E$ acts trivially on the Dynkin diagram of $G$.
Let's say that a finite set $\cR$ of primes is \tit{$G$-good} if for each conjugacy class $C$ of $\Gal(E|\Q)$, there exists $\ell_C\in\cR$ such that
\begin{enumerate}
\item
$\ell_C$ is unramified in $E|\Q$.
\item
the Frobenius conjugacy class of $\ell_C$ in $\Gal(E|\Q)$ is $C$.
\item
$G_{\Q_{\ell_C}}$ is quasisplit.
(It is well known that $G_{\Q_\ell}$ is quasisplit for all but finitely many $\ell$.)
\end{enumerate}
\end{defn}

\begin{mainthm}[Theorem \ref{thm:main} below]
\label{thm:A}
Let $G$ be a connected reductive group over $\Q$, and let $\{\rho_\ell\cln\Gamma\to G(\Q_\ell)\}_{\ell\in\cL}$  be an abstract compatible collection of $G$-representations such that $\cL$ is a set of Dirichlet density $1$.
For each $\ell\in\cL$, let $M_\ell\ceq\ol{\Img(\rho_\ell)}{}^\Zar$.
If there exists a $G$-good set $\cR\subseteq\cL$ such that $M_\ell=G_{\Q_\ell}$ for each $\ell\in\cR$, then
$
\{
\ell\in\cL
\cln
M_\ell
=
G_{\Q_\ell}
\}
$
has Dirichlet density $1$.
\end{mainthm}

\begin{rmk}
\label{rmk:main}
\mbox{}
\begin{enumerate}
\item
In particular, when the $\Gal_\Q$-action on the Dynkin diagram is trivial, the conclusion of Theorem \ref{thm:A} becomes nicer: If $M_\ell=G_{\Q_\ell}$ for a \tit{single} prime $\ell_0$ for which $G_{\Q_{\ell_0}}$ is quasisplit, then
$
\{
\ell\in\cL
\cln
M_\ell
=
G_{\Q_\ell}
\}
$
has Dirichlet density $1$.
\item
Since Theorem \ref{thm:A} applies to any reductive $G$ over $\Q$, we can use a restriction-of-scalars argument to deduce a version with ``coefficients in any number field''.
Specifically, let $\{\rho_\lambda\cln\Gamma\to G(K_\lambda)\}_{\lambda\in\cL}$ be an abstract compatible collection, where $K$ is a number field, $G$ is a connected reductive group over $K$, and $\cL$ is a set of primes of $K$.
Then we obtain a collection $\{\rho_\ell\cln\Gamma\to G_0(\Q_\ell)\}_{\ell\in\cL_0}$, where $G_0\ceq\op{Res}_{K|\Q}(G)$ (Weil's restriction of scalars) and $\cL_0$ is the set of rational primes all of whose prime divisors in $K$ live in $\cL$.
It is easily checked that $\{\rho_\ell\}_{\ell\in\cL_0}$ is again compatible.
So Theorem \ref{thm:A} may be invoked:
If $\cL_0$ has Dirichlet density $1$, and there exists a $G_0$-good set $\cR$ such that $M_\lambda=G_{K_\lambda}$ for all $\lambda$ dividing an element of $\cR$, then there exists a set $\cL_0'\subseteq\cL_0$ of Dirichlet density $1$ such that $M_\lambda=G_{K_\lambda}$ for all $\lambda$ dividing an element of $\cL_0'$.
\item
When $G$ is semisimple, Serre has shown, using an argument with Lie algebras, that a Zariski-dense subgroup of $G(\Q_\ell)$ must be \tit{open} for the $\ell$-adic topology \cite[Corollary to Proposition 2]{serre-p-divisible-groups}.
In the setting of Theorem \ref{thm:A}, we can use the compatibility of $\{\rho_\ell\}_\ell$ to invoke \cite[Theorem 3.17]{larsen-maximality} and obtain more: for all $\ell$ in a set of Dirichlet density $1$, $\Img(\rho_\ell)$ is not just open in $G(\Q_\ell)$ but ``close'' to being a maximal compact subgroup in the following sense.
Consider the natural maps
\[
G\xrightarrow\sigma
G/\mr R(G)\xleftarrow\tau
G^\mr{sc},
\]
where $\mr R(G)$ is the radical of $G$ and $G^\mr{sc}$ is the simply connected cover of $G/\mr R(G)$;
then $\tau^{-1}(\sigma(\Img(\rho_\ell)))$ is a hyperspecial subgroup of $G^\mr{sc}(\Q_\ell)$, meaning that $G^\mr{sc}_{\Q_\ell}$ spreads out to a reductive group over $\Z_\ell$, and $\tau^{-1}(\sigma(\Img(\rho_\ell)))=G^\mr{sc}(\Z_\ell)$ for some such spreading-out.
\item
Let $\{\rho_\ell\cln\Gamma\to G(\Q_\ell)\}_{\ell\in\cL}$ be any abstract compatible collection of $G$-representations, where $G$ is semisimple, and assume that $M_\ell$ is of maximal rank in $G_{\Q_\ell}$ for one $\ell\in\cL$.
Also, fix a faithful $\Q$-representation $\xi\cln G\to\GL_r$.
By the result of Serre mentioned above \cite[\textsection3]{serre-letter}, $M_\ell$ is then of maximal rank in $G_{\Q_\ell}$ for all $\ell\in\cL$.
It follows from \cite[Theorem 7.1]{dynkin} that if $\xi$ is absolutely irreducible, then $\xi\circ\rho_\ell$ is absolutely irreducible precisely when $M_\ell=G_{\Q_\ell}$.
Thus Theorem \ref{thm:A} can be interpreted as a ``transport of irreducibility'' result for the very special class of compatible collections of $\GL_r$-representations of the form $\{\xi\circ\rho_\ell\}_\ell$.
\item
On the proof:
We rely heavily on the ideas developed in \cite{larsen-pink} and can think of no better introduction to them than the first three pages of \tit{loc.\ cit.}, which in particular walks the reader through the proof of Theorem \ref{thm:A} for $G=\SL_2$.
\end{enumerate}
\end{rmk}

\subsection{Second main result and application to Shimura varieties}

Our second result is a straightforward application of Theorem \ref{thm:A}.
In the statement, we write $\rho_{\ell,x}$ for the restriction of $\rho_\ell$ to $\pi_1(\{x\})$ (see (\tit d) of \textsection\ref{subsec:notation} below) and $M_{\ell,x}\ceq\ol{\Img(\rho_{\ell,x})}{}^\Zar$.

\begin{mainthm}[Theorem \ref{thm:main-2} below]
\label{thm:B}
Let $X$ be a connected normal scheme of finite type over $\Z$ such that $\dim(X_\Q)\geq1$, and let $\{\rho_\ell\cln\pi_1(X_{\Z[1/\ell]})\to G(\Q_\ell)\}_{\ell\in\cL}$ be a compatible collection of $G$-representations, where $G$ is a connected reductive group over $\Q$ and $\cL$ is a set of Dirichlet density $1$.
Assume that $M_\ell=G_{\Q_\ell}$ for each $\ell\in\cL$.

Suppose $x$ is a closed point of $X_\Q$ with the following property:
There exists a positive integer $N$ such that
$x$ extends to an element of $X(\ol\Z[1/N])$ and a $G$-good set $\cR\subseteq\cL$ of primes not dividing $N$ such that $M_{\ell,x}=G_{\Q_\ell}$ for each $\ell\in\cR$.
Then $\{\ell\in\cL\cln M_{\ell,x}=G_{\Q_\ell}\}$ has Dirichlet density $1$.

Moreover, there exists a positive integer $d$ and infinitely many $x$ satisfying the hypotheses of the previous paragraph and also satisfying $[\kappa_x:\Q]\leq d$.
\end{mainthm}

\begin{rmk}
\mbox{}
\begin{enumerate}
\item
The bound on the residue degree is explicit:
If $X_\Q$ admits a finite morphism to $\A^n_\Q$ of degree $d$, then this number works in the final sentence of the statement of Theorem \ref{thm:B}.
In particular, if $X_\Q$ admits the structure of a rational $k$-variety, where $k$ is a number field, then we may take $\kappa_x=k$.
Moreover, if $X_\Q\to\A^n_\Q$ spreads out to a finite map $X_{\Z[1/N]}\to\A^n_{\Z[1/N]}$, then we can ask that $x$ extend to an element of $X(\ol\Z[1/N])$.
\item
On the proof:
The first part of Theorem \ref{thm:B} is an immediate consequence of Theorem \ref{thm:A}.
The second part, the abundance of specializations with ``big monodromy'', follows from a variant of Serre's version of Hilbert's irreducibility theorem for profinite groups \cite[\textsection10.6]{serre-mordell-weil}.
\end{enumerate}
\end{rmk}

\begin{cor}
\label{cor:shimura}
Let $(G,X)$ be a Shimura datum such that $\rank_\R(G^\mr{ad})\geq2$, let $S$ be a geometrically connected component \textnormal(defined over a number field\textnormal) of a Shimura variety attached to $(G,X)$, and let $\{\rho_\ell\cln\pi_1(S)\to G^\mr{ad}(\Q_\ell)\}_\ell$ be the adjoint projections of the ``canonical'' $\ell$-adic local systems on $S$, defined as in \textnormal{\cite[\textsection4]{cadoret-kret}} or \textnormal{\cite[\textsection\textsection3.1--3.2]{klevdal-patrikis}}.
Then there exists a positive integer $d$ and infinitely many closed points $x$ of $S$ satisfying $[\kappa_x:\Q]\leq d$ such that $\{\ell\cln M_{\ell,x}=G_{\Q_\ell}\}$ has Dirichlet density $1$.
\end{cor}

\begin{proof}
By recent work of Klevdal--Patrikis \cite[Theorem 1.3]{klevdal-patrikis}, the $\rho_\ell$ extend to an integral model of $S$, and the extensions are compatible in the sense described above, allowing us to invoke Theorem \ref{thm:B}.
\end{proof}

\begin{thing}
Due to versions of Tate's conjecture for Abelian varieties proven by Faltings, \tit{much} stronger results are known for $(G,X)$ of Abelian type; see e.g.\ the work of Cadoret--Kret \cite[Theorem A]{cadoret-kret}.
Thus Corollary \ref{cor:shimura} is novel only when $(G,X)$ is not of Abelian type, in which case the $\rho_\ell$ are not known, but conjectured, to be of geometric origin.
As a concrete example, we obtain from these ``non-Abelian Shimura varieties'' compatible collections of representations $\{\Gal_F\to G(\Q_\ell)\}$, where $F$ is a number field and $G$ is an adjoint group of type $\mr E_{6(-14)}$ or $\mr E_{7(-25)}$, with Zariski-dense image along a set of primes of Dirichlet density $1$.
\end{thing}

\subsection{Notation}
\label{subsec:notation}

Generalities:
\begin{enumerate}
\item
If $K$ is a field, $K^\mr{s}$ denotes a separable closure of $K$, and $\Gal_K$ always means the absolute Galois group of $K$, i.e.\ $\Gal(K^\mr{s}|K)$.
\item
Given a scheme $X$ over a ring $A$ and a ring morphism $A\to B$, we write $X_B\ceq X\times_{\Spec(A)}\Spec(B)$.
\end{enumerate}
Algebraic groups:
\begin{enumerate}[resume]
\item
If $\Gamma$ is a profinite group and $G$ is an algebraic group over $\Q_\ell$, then by a ``$G$-representation'' we mean a homomorphism $\Gamma\to G(\Q_\ell)$ which is continuous for the profinite topology on $\Gamma$ and the $\ell$-adic topology on $G(\Q_\ell)$.
(When $G=\GL_n$, we may omit ``$G$'' from the terminology.)
\item
Given an algebraic group $G$ over $\Q_\ell$ and $S\subseteq G(\Q_\ell)$, we write $\ol S{}^\Zar$ for the Zariski closure of $S$ (a closed subscheme of $G$ which itself is an algebraic group).
\item
Given a reductive group $G$, we write $G^\mr{ad}\ceq G/\mr Z(G)$ for the adjoint group attached to $G$ and $G^\mr{Ab}\ceq G/[G,G]$ the Abelianization.
Given an element $g\in G(K)$, we let $g_\ssimp$ denote the semisimple part of $g$ in its Jordan decomposition.
\end{enumerate}
Fundamental groups of schemes:
\begin{enumerate}[resume]
\item
For $X$ a connected scheme, we denote by $\pi_1(X)$ the \'etale fundamental group of $X$, leaving the basepoint implicit.
Given a point $x\in X$, we get a homomorphism $\psi_x\cln\Gal_{\kappa_x}\to\pi_1(X)$ by identifying $\Gal_{\kappa_x}$ with the \'etale fundamental group of the scheme $\{x\}$; this $\psi_x$ is, of course, well defined only up to conjugation by $\Gal_{\kappa_x}$.
\item
Given a group homomorphism $\rho\cln\pi_1(X)\to\Pi$ and a point $x\in X$, let $\kappa_x$ denote the residue field of $x$ and $\rho_x\cln{\Gal_{\kappa_x}}\to\Pi$ the composition $\rho\circ\psi_x$.
\item
When $s$ is a point of $X$ with finite residue field, we let $\Frob_s$ be the conjugacy class of $\psi_s(\Frob_{\kappa_s})$ in $\pi_1(X)$, where $\Frob_{\kappa_s}$ is the Frobenius automorphism of $\kappa_s$; the elements of $\Frob_s$ are the ``Frobenius elements of $s$ in $\pi_1(X)$''.
\end{enumerate}

\subsection{Outline of the paper}

In \textsection\ref{sec:bg}, we provide the necessary background on reductive groups and their maximal tori.
In \textsection\ref{subsec:compatible}, we define and study ``abstract'' compatible collections of $G$-representations.
This allows us to state and prove Theorem \ref{thm:A} in \textsection\ref{subsec:thm-A}.
Finally, we give a version of Hilbert's irreducibility theorem in \textsection\ref{subsec:hilbert-irreducibility} and discuss Theorem \ref{thm:B} in \textsection\ref{subsec:thm-B}.

\subsection{Acknowledgements}

We are very happy to thank our PhD advisor Stefan Patrikis for his persistent generosity and encouragement and for sharing with us many useful references.
Among these, we are particularly indebted to Michael Larsen and Richard Pink for the results and ideas contained in their marvelous paper \cite{larsen-pink}.
We are also grateful Christian Klevdal and Stefan Patrikis for their paper \cite{klevdal-patrikis}, which gave birth to this one.
Finally, we thank Ishan Banerjee for his interest in the project and helpful comments on the manuscript.

This material is based upon work supported by the National Science Foundation under Grant No.\ DMS-2231565.

\section{Group-theoretic background}
\label{sec:bg}

\subsection{Classification of quasisplit reductive groups and their maximal tori}
\label{subsec:groups-and-tori}

\noindent
In this \textsection\ref{subsec:groups-and-tori} we essentially summarize \cite[3.1--3.6]{larsen-pink}.
Let $K$ be a field with a fixed separable closure $\Ks$. 

\begin{thing}
\label{thing:G-setup}
Fix a connected reductive\footnote{In fact, with the exception of Lemma \ref{lem:super-jordan}, the final arguments will reduce immediately to the case when $G$ is semisimple, so the reader loses little in assuming this to be so for all of \textsection\ref{sec:bg} and replacing all root data with root systems.}
group $G$ over $K$ and a maximal torus $T_0$ of $G_{\Ks}$ (which need not be defined over $K$, though we will later assume it to be).
Let $\Psi_0$ be the root datum attached to $T_0$.
Denote by $\W(\Psi_0)$ and $\Aut(\Psi_0)$, respectively, the Weyl and automorphism groups of $\Psi_0$ in the sense of root data, and by $\Out(\Psi_0)$ the quotient $\Aut(\Psi_0)/{\W(\Psi_0)}$.
Finally, for any subset $\Omega\subseteq\Aut(\Psi_0)$ stable under conjugation by $\W(\Psi_0)$, we will denote by $[\Omega]$ the set of orbits of the conjugaction action of $\W(\Psi_0)$ on $\Omega$.
\end{thing}

\begin{thing}
\label{thing:tori-type}
Let $T$ be any maximal torus of $G$.
The Galois action on $T_\Ks$ induces a homomorphism $\phi_T\cln{\Gal_K}\to\GL(\Xup(T_{K^{\mr s}}))$.
If we fix an element $a\in G(\Ks)$ satisfying $aT_\Ks a^{-1}=T_0$, it induces an isomorphism $\theta_a\cln{\GL(\Xup(T_{K^{\mr s}}))}\iso\GL(\Xup(T_0))$, and the image of $\theta_a\circ\phi_T$ lands in $\Aut(\Psi_0)$.
The isomorphism $\theta_a$ depends on $a$ only up to $\W(\Psi_0)$-conjugation, so the $\W(\Psi_0)$-conjugacy class $[\phi_T]$ of $\theta_a\circ\phi_T$ depends only on $T$ (and the fixed objects $G$ and $T_0$).
Below, we will conflate $[\phi_T]$ with the function $\Gal_K\to[\Aut(\Psi_0)]$ valued in the set of $\W(\Psi_0)$-conjugacy classes of $\Aut(\Psi_0)$.

The composition
\[
\Gal_K
\xrightarrow{\theta_a\circ\phi_T}
\Aut(\Psi_0)
\onto
\Out(\Psi_0)
\]
is therefore also well defined, and one checks that it is \tit{independent of $T$}.
Denote it by $\phi_G$.
(This keeps track of action on $G^\mr{Ab}$ and the ``$*$-action'' on the Dynkin diagram induced by $\Gal_K$.)
Thus each maximal torus of $G$ corresponds to a \tit{$\W(\Psi_0)$-conjugacy class of lifts} of $\phi_G$.
\end{thing}

\begin{thing}
\label{thing:conjugate-over-K}
If $T$ and $T'$ are maximal tori of $G$ which are conjugate \tit{by an element of $G(K)$}, then $[\phi_T]=[\phi_{T'}]$.
Indeed, we have $\theta_b\circ\phi_{T'}=\phi_T$ for any $b\in G(K)$ satisfying $bT'b^{-1}=T$.
\end{thing}

\begin{thing}
\label{thing:exists-reductive}
Given an abstract root datum $\Psi$ and a homomorphism $\Gal_K\to\Out(\Psi)$, there exists a \tit{quasisplit} connected reductive group over $K$, unique up to isomorphism, realizing this data in the sense of (\ref{thing:G-setup}--\ref{thing:tori-type}).
\end{thing}

\noindent
The following key result rests ultimately on Steinberg's theorem that every rational conjugacy class in a quasisplit simply connected reductive group contains a rational element.

\begin{lem}[{\cite[Lemma 3.6]{larsen-pink}}]
\label{lem:exists-torus}
Let $G$ be a quasisplit reductive group over $K$.
Then for any $\W(\Psi_0)$-conjugacy class $[\phi]$ of lifts of $\phi_G$, there is a maximal torus $T$ of $G$ such that $[\phi_T]=[\phi]$.
\end{lem}

\subsection{Splitting field of a conjugacy class}
\label{subsec:splitting}

\begin{thing}
Let $g\in G(\Ks)$ be an element whose $G(\Ks)$-conjugacy class $[g]$ is defined over $K$ (i.e.\ stable under the $\Gal_K$-action on $G(\Ks)$).
For any $K$-representation $\xi\cln G\to\GL_d$, the $\GL_d(\Ks)$-conjugacy class of $\xi(g)$ is defined over $K$, so the characteristic polynomial $P$ of $\xi(g)$ has coefficients in $K$.
If $\xi$ is faithful, we will call the splitting field $F$ of $P$ over $K$ the \tit{splitting field of $[g]$} over $K$.
It is well defined because if $T$ is a maximal torus of $G_\Ks$ containing the semisimple part $g_\ssimp$ of $g$, then $F$ is the subfield of $\Ks$ generated by the values $\chi(g_\ssimp)$ where $\chi\in \Xup(T)$.
\end{thing}

\begin{thing}
\label{thing:pre-phi-T-g}
We now describe the action of $\Gal_K$ on the roots of $P$ in a way intrinsic to $G$.
To do so, we may assume (after conjugating $g$ and replacing it by $g_\ssimp$) that it is contained in a maximal torus $T$ \tit{defined over $K$}.
For simplicity, we will consider the case when $G$ is an adjoint group (i.e.\ has trivial center), $\xi$ is its adjoint representation, and the values $\chi_1(g),\dots,\chi_m(g)$ are distinct, where $\chi_1,\dots,\chi_m$ are the nontrivial characters of $T_\Ks$ making up $\xi|_T$.
(Such $g$ are called \tit{strongly regular}.)

In this case, the values $\chi_1(g),\dots,\chi_m(g)$ are the roots of $P$ different from $1$, so the Galois action on the roots of $P$ permutes the set of characters $\{\chi_1,\dots,\chi_m\}$.
\end{thing}

\begin{lem}
\label{lem:phi-T-g}
Assume $G$ is adjoint and $g\in G(\Ks)$ is a strongly regular element contained in $T(\Ks)$ for a maximal torus $T$ of $G$.
Then the $\Gal_K$-action on $\{\chi_1,\dots,\chi_m\}$ as in \eqref{thing:pre-phi-T-g} extends uniquely to a homomorphism $\phi_{T,g}\cln{\Gal_K}\to\GL(\Xup(T_\Ks))$ given by $\sigma\cdot\chi={}^{\sigma\!}\chi\circ\inn(a_\sigma)$ for some $a_\sigma\in\op N_G(T)(\Ks)$.
\end{lem}

\begin{proof}
Uniqueness presents no trouble, because the characters $\chi_1,\dots,\chi_m$ span $\Xup(T_\Ks)$ by the hypothesis on $G$.
Now fix $\sigma\in\Gal_K$, and find $a_\sigma\in G(\Ks)$ such that $\sigma (g)=a_\sigma ga_\sigma^{-1}$.
Then if $\chi\in \Xup(T_\Ks)$, we have $\sigma(\chi(g))={}^{\sigma\!}\chi(\sigma(g))={}^{\sigma\!}\chi(a_\sigma ga_\sigma^{-1})$.
Now $a_\sigma\in\op N_G(T)(\Ks)$ because $a_\sigma T_\Ks a_\sigma^{-1}$ is the unique maximal torus of $G_\Ks$ containing $a_\sigma g a_\sigma^{-1}=\sigma(g)$ by the regularity of $\sigma(g)$, but $T_\Ks$ already contains $\sigma(g)$ because $T$ is defined over $K$.
Finally, $\chi\mapsto{}^{\sigma\!}\chi\circ\inn(a_\sigma)$ must permute the set $\{\chi_1,\dots,\chi_m\}$ because these are the roots of $G$ relative to $T$; thus if $\sigma(\chi_i(g))=\chi_j(g)$, then ${}^{\sigma\!}\chi_i\circ\inn(a_\sigma)=\chi_j$ by the strong regularity of $g$, so $\phi_{T,g}$ extends the desired $\Gal_K$-action.
\end{proof}

\begin{thing}
\label{thing:phi-T-g}
Two easy observations about the $\phi_{T,g}$ of Lemma \ref{lem:phi-T-g}:
\begin{enumerate}
\item
Since $\phi_{T,g}(\sigma)$ differs from $\phi_T(\sigma)$ by $\inn(a_\sigma)$, the projection of $\theta_a\circ\phi_{T,g}$ to $\Out(\Psi_0)$ is again equal to $\phi_G$ for any $a\in G(\Ks)$ satisfying $aT_\Ks a^{-1}=T_0$.
\item
If $g\in G(K)$, then $\phi_{T,g}=\phi_T$.
\end{enumerate}
Following the notation of \eqref{thing:tori-type}, we write $[\phi_{T,g}]$ for the $\W(\Psi_0)$-conjugacy class of $\theta_a\circ\phi_{T,g}$.
\end{thing}

\noindent
The following easy property of the $\phi_{T,g}$ generalizes \eqref{thing:conjugate-over-K} and will be crucial later on.

\begin{lem}
\label{lem:phi-T-g-commutes}
Assume $G$ is adjoint and $g,g'\in G(\Ks)$ are strongly regular elements contained, respectively, in maximal tori $T,T'$ of $G$.
Suppose $b\in G(\Ks)$ satisfies $bg'b^{-1}=g$.
Then $\theta_b\circ\phi_{T',g'}=\phi_{T,g}$.
\end{lem}

\begin{proof}
Let $\sigma\in\Gal_K$ and $\chi\in \Xup(T_\Ks)$.
Then $\phi_{T,g}(\sigma)(\chi)={}^{\sigma\!}\chi\circ\inn(a_\sigma)$, while
\begin{align*}
(\theta_b\circ\phi_{T',g'})(\sigma)(\chi)
&=
\phi_{T',g'}(\sigma)(\chi\circ\inn(b))\circ\inn(b^{-1})
\\&=
{}^{\sigma\!}\chi\circ\inn(\sigma(b)a_\sigma'b^{-1}).
\end{align*}
But $(\sigma(b)a_\sigma'b^{-1})g(\sigma(b)a_\sigma'b^{-1})^{-1}=a_\sigma ga_\sigma$, so $\inn(\sigma(b)a_\sigma'b^{-1})$ and $\inn(a_\sigma)$ induce the same automorphism of $\GL(\Xup(T_\Ks))$ by the regularity of $g$.
\end{proof}

\subsection{Setup of the main theorems (case \texorpdfstring{\boldmath$K=\Q\hspace{-10pt}\Q\hspace{-10pt}\Q\hspace{-10pt}\Q\hspace{-10pt}\Q\hspace{-10pt}\Q$}{K=Q})}

We continue with the notation of \textsection\textsection\ref{subsec:groups-and-tori}--\ref{subsec:splitting}, but this time with $K=\Q$.

\begin{thing}
\label{thing:main-thm-G}
Fix an algebraic closure $\ol\Q$ of $\Q$.
As above we have a connected reductive group $G$, and from now on we (for convenience) assume $T_0$ is a maximal torus of $G$ rather than $G_{\ol\Q}$ (so $\Psi_0$ is now the root datum attached to $T_{0,\ol\Q}$, etc.).
Let $E|\Q$ be the splitting extension of $\phi_G$, i.e.\ the extension for which $\Ker(\phi_G)=\Gal_E$.
\end{thing}

\begin{thing}
\label{thing:embeddings}
For each prime $\ell$, fix an algebraic closure $\ol{\Q_\ell}$ of $\Q_\ell$.
Let $\iota_\ell$ be a field embedding $\ol\Q\into\ol{\Q_\ell}$.
It induces an embedding $\Gal_{\Q_\ell}\into\Gal_\Q$, and given a torus $T$ over $\Q$, maps $ \Xup(T_{\ol\Q})\to \Xup(T_{\ol{\Q_\ell}})$ and ${\GL(\Xup(T_{\ol\Q}))\to\GL(\Xup(T_{\ol{\Q_\ell}}))}$ via base change.
By choosing an isomorphism $T_{\ol\Q}\cong\G_{\mr m}^r$, one sees that the latter are isomorphisms (but depend on $\iota_\ell$).
If $T$ is a maximal torus of $G$, then $\iota_\ell$ induces isomorphisms
\begin{equation}
\label{eq:compatibility-0}
\Aut(\Psi)
\iso
\Aut(\Psi_\ell)
\quad
\text{and}
\quad
\W(\Psi)
\iso
\W(\Psi_\ell),
\end{equation}
where $\Psi$ and $\Psi_\ell$ are, respectively, the root data attached to $T_{\ol\Q}$ and $T_{\ol{\Q_\ell}}$.

Having picked a embedding $\iota_\ell$, we are safe to forget it in the notation and \tit{identify} all the $\Q$- and $\Q_\ell$-versions of the above objects, because of \eqref{eq:compatibility-0} and because the following diagrams commute:
\[
\begin{tikzcd}
\Gal_{\Q_\ell}\arrow[r,"\phi_{T_{\Q_\ell}}"]\arrow[d,hook]&
\GL(\Xup(T_{\ol{\Q_\ell}}))\arrow[d,"\sim"]\\
\Gal_\Q\arrow[r,"\phi_T"]&
\GL(\Xup(T_{\ol\Q}))
\end{tikzcd}
\]
where $T$ is any torus over $\Q$,
\begin{equation}
\label{eq:compatibility-2}
\begin{tikzcd}
[column sep=large]
\Gal_{\Q_\ell}\arrow[r,"\phi_{T_{\Q_\ell},\iota_\ell(h)}"]\arrow[d,hook]&
\GL(\Xup(T_{\ol{\Q_\ell}}))\arrow[d,"\sim"]\\
\Gal_\Q\arrow[r,"\phi_{T,h}"]&
\GL(\Xup(T_{\ol\Q}))
\end{tikzcd}
\end{equation}
where $T$ is any maximal torus of $G$ and $h\in T(\ol\Q)$ any regular semisimple element, and
\begin{equation}
\label{eq:compatibility-3}
\begin{tikzcd}
\GL(\Xup(T'_{\ol{\Q_\ell}}))\arrow[r,"\theta_{\iota_\ell(a)}"]\arrow[d,"\sim"']&
\GL(\Xup(T_{\ol{\Q_\ell}}))\arrow[d,"\sim"]\\
\GL(\Xup(T'_{\ol\Q}))\arrow[r,"\theta_a"]&
\GL(\Xup(T_{\ol\Q}))
\end{tikzcd}
\end{equation}
where $T,T'$ are any maximal tori of $G$ and $a\in G(\ol\Q)$ is any element satisfying $aT'_{\ol\Q}a^{-1}=T_{\ol\Q}$.
Thus, for example, if $T_\ell$ is a maximal torus of $G_{\Q_\ell}$, then (via $\iota_\ell$) we view $[\phi_{T_\ell}]$ as a function $\Gal_{\Q_\ell}\to[\Aut(\Psi_0)]$ and its domain as a subgroup of $\Gal_\Q$.
\end{thing}

\begin{thing}
\label{thing:frobs}
Given a prime $\ell$, we will let $\Frob_\ell\subseteq\Gal_{\Q_\ell}$ denote the Frobenius coset of $\Q_\ell$.
In accordance with \eqref{thing:embeddings}, if $\iota_\ell$ has been chosen, we view $\Frob_\ell$ as a subset of $\Gal_\Q$.
Then for a finite Galois extension $F|\Q$ unramified at $\ell$, there is a Frobenius \tit{element} in $\Gal(F|\Q)$ defined by $\Frob_\ell$.
Moreover if $\sigma$ is any element of the Frobenius conjugacy class of $\ell$ in $\Gal(F|\Q)$, one can pick $\iota_\ell$ so that the image of $\Frob_\ell$ in $\Gal(F|\Q)$ is $\{\sigma\}$.
\end{thing}

\begin{thing}
\label{thing:super-jordan-setting}
Fix a prime $\ell$ unramified in $E|\Q$ and an embedding $\iota_\ell$.
Let $\Omega$ be the preimage of $\phi_G(\Frob_\ell)$ in $\Aut(\Psi_0)$; it is a coset of $\W(\Psi_0)$.
If $T$ is an \tit{unramified} maximal torus of $G_{\Q_\ell}$, i.e.\ $\phi_T$ factors through the Galois group of the maximal unramified extension of $\Q_\ell$, then $[\phi_T](\Frob_\ell)$ is a well defined element $\omega\in[\Omega]$.\footnote{
Recall that $[\Omega]$ is the set of orbits under the action of $\W(\Psi_0)$ on $\Omega$ by conjugation.
This makes sense because $\W(\Psi_0)$ is normal in $\Aut(\Psi_0)$.
}
Say, in this case, that the torus $T$ \tit{corresponds to $\omega$}.
\end{thing}

\begin{lem}
\label{lem:super-jordan}
In the context of \eqref{thing:super-jordan-setting}, suppose $M$ is a reductive subgroup of $G_{\Q_\ell}$ such that for each $\omega\in[\Omega]$, there exists an unramified maximal torus of $G_{\Q_\ell}$ which is included in $M$ and corresponds to $\omega$.
Then $M=G_{\Q_\ell}$.
\end{lem}

\begin{proof}
The hypothesis remains true after conjugating the torus $T_0$, so we may assume that $T_{0,\Q_\ell}\subseteq M$.
Let $\Psi_0'$ be the root datum attached to the pair $(T_{0,\ol{\Q_\ell}},M_{\ol{\Q_\ell}})$, and let $\Omega'$ be the preimage of $\{\phi_M(\Frob_\ell)\}$ in $\Aut(\Psi_0')$.
Then $\W(\Psi_0')\subseteq\W(\Psi_0)$ and $\Omega'\subseteq\Omega$, and the hypothesis implies that $\Omega'$ meets every element of $[\Omega]$.
In particular, the set
\[
\{
\text{characteristic polynomial of $\alpha$ on $\Xup(T_{0,\ol{\Q_\ell}})$}\cln\alpha\in\Omega
\}
\]
equals the corresponding set with $\Omega'$ in place of $\Omega$.
Thus by \cite[Theorem 2.1]{larsen-pink}, the Weyl group of $M$ is isomorphic to that of $G$.
Since the Weyl group and character lattice together determine the roots up to rational multiples \cite[Chapter IV, 1.5, Theorem 2(\tit{iv})]{bourbaki-lie}, every root of $G$ is also a root of $M$, hence $M=G_{\Q_\ell}$.
\end{proof}

\begin{rmk}
In the previous proof, we passed from $\W(\Psi_0)$-conjugacy classes in $\Omega$ to the much coarser characteristic polynomials of Frobenius acting on character lattices.
Thus one expects that we should be able to avoid the use of \cite[Theorem 2.1]{larsen-pink}, whose proof is somewhat complicated and relies on very explicit computations with semisimple groups (ultimately resorting to data tabulated in the \textsc{Atlas} of Finite Groups).
Indeed, when $\Gal_\Q$ acts on the Dynkin diagram of $G$ only by permuting its connected components (e.g.\ if $G$ is split or has no factors of type $\mr A_n$ ($n\geq2$), $\mr D_n$, or $\mr E_6$), one can replace \tit{loc.\ cit.\ }by a short abstract argument, but we did not find one that works in general.
(In fact, when the action on the Dynkin diagram is trivial, one has only to observe that the only subgroup of $\W(\Psi_0)$ which meets every conjugacy class of $\W(\Psi_0)$ is $\W(\Psi_0)$ itself.)
\end{rmk}

\section{Theorem \texorpdfstring{\ref{thm:A}}{A}}
\label{sec:thm-A-pf}

We will freely use the notation of \textsection\ref{sec:bg}.

\subsection{Compatible collections}
\label{subsec:compatible}

The following definition generalizes \cite[Definition 6.5]{larsen-pink}, in which $G=\GL_r$.

\begin{defn}
\label{defn:compatible}
An \tit{F-group} is a pair $(\Gamma,\cF)$, where $\Gamma$ is a profinite group and $\cF$ is a dense subset of $\Gamma$.
A set $\{\rho_\ell\cln\Gamma\to G(\Q_\ell)\}_{\ell\in\cL}$ of $G$-representations, where $\cL$ is a set of primes, is \tit{compatible} (relative to $\cF$) if there exists a function $\ff\mapsto\cS_\ff$ from $\cF$ to the set of finite subsets of $\cL$ with the following properties:
\begin{enumerate}
\item
for each $\ff\in\cF$, there exists $g\in G(\ol\Q)$ whose conjugacy class $[g]$ is defined over $\Q$ (i.e.\ $\Gal_\Q$-stable) such that $\rho_\ell(\ff)_\ssimp$ is conjugate to $g$ in $G(\ol{\Q_\ell})$ for each $\ell\in\cL\setminus\cS_\ff$.
(This property of $g$ does not depend on the choice of embeddings $\iota_\ell$.)
\item
for each $\ell_1,\dots,\ell_n\in\cL$, the set $\cF^{(\ell_1,\dots,\ell_n)}\ceq\{\ff\in\cF\cln\ell_1,\dots,\ell_n\notin\cS_\ff\}$ is dense in $\Gamma$.
\end{enumerate}
\end{defn}

\noindent
Before proving the main theorem, we transform the $\rho_\ell$ into information about maximal tori of the $M_\ell$, analogously to \cite[7.4--7.5]{larsen-pink}, at least in a special case.
This lemma will be instrumental in the proof of the main theorem, and the use of the functions $[\phi_\fT]$ constructed herein is a key difference between the arguments of this paper and those of \cite{larsen-pink}.

\begin{lem}
\label{lem:torus-info}
Let $\{\rho_\ell\cln\Gamma\to G(\Q_\ell)\}_{\ell\in\cL}$ a compatible collection of $G$-representations of an F-group, where $G$ is an adjoint group over $\Q$.
For each $\ell\in\cL$, let $M_\ell\ceq\ol{\Img(\rho_\ell)}{}^\Zar$, and assume $M_\ell$ is reductive and of maximal rank in $G_{\Q_\ell}$.
Fix $\iota_\ell$ for each $\ell$ as in \eqref{thing:embeddings}.

Let $\fT\ceq\{(\ell_1,T_1),\dots,(\ell_n,T_n)\}$ be a set of pairs consisting of distinct primes $\ell_1,\dots,\ell_n\in\cL$ and a maximal torus $T_i$ of $M_{\ell_i}$ for each $i$.
Then there exist a $\W(\Psi_0)$-conjugacy class $[\phi_\fT]$ of maps ${\Gal_\Q}\to\Aut(\Psi_0)$ lifting $\phi_G$ and a maximal torus $T_{\fT,\ell}$ of $M_\ell$ for all but finitely many $\ell\in\cL$ such that
$[\phi_{T_{\fT,\ell}}]=[\phi_\fT]|_{\Gal_{\Q_\ell}}$ for every such $\ell$ and
$T_{\fT,\ell_i}=T_i$ for each $i$.
Moreover, $[\phi_\fT]$ and the $T_{\fT,\ell}$ do not depend on the $\iota_\ell$.
\end{lem}

\begin{proof}
By \cite[Proposition 7.3]{larsen-pink}, we can find $\ff\in\cF^{(\ell_1,\dots,\ell_n)}$ such that $\rho_{\ell_i}(\ff)_\ssimp$ is, for each $i$, conjugate \tit{in $G(\Q_{\ell_i})$} to an element of $T_i(\Q_{\ell_i})$ and is strongly regular in the sense of \eqref{thing:phi-T-g}.
Therefore, letting $T_{\fT,\ell}$ be the unique maximal torus of $M_\ell$ containing $\rho_\ell(\ff)_\ssimp$ for each $\ell\in\cL\setminus\cS_\ff$, we have $[\phi_{T_{\fT,\ell}}]=[\phi_{T_i}]$ for each $i$ by \eqref{thing:conjugate-over-K}.
Using (\ref{defn:compatible}.\tit a), find an element $g\in G(\ol\Q)$ whose conjugacy class $[g]$ is defined over $\Q$ and such that $g$ is conjugate in $G(\ol{\Q_\ell})$ to $\rho_\ell(\ff)_\ssimp$ for each $\ell\in\cL\setminus\cS_\ff$.
After replacing $g$ by a conjugate, we may assume it is contained in $T(\ol\Q)$ for a unique maximal torus $T$ of $G$ (defined over $\Q$!).
By Lemma \ref{lem:phi-T-g-commutes}, $[\phi_{T_{\fT,\ell}}]=[\phi_{T,g}]|_{\Gal_{\Q_\ell}}$, so $[\phi_\fT]\ceq[\phi_{T,g}]$ works.

Let us explain the preceding sentence in full details.
Consider the following diagram, where $a\in G(\ol\Q)$ and $a_\ell\in G(\ol{\Q_\ell})$ are any elements which make sense in the diagram, and $b\in G(\ol{\Q_\ell})$ satisfies $b\cdot\rho_\ell(\ff)_\ssimp\cdot b^{-1}=g$:
\[
\begin{tikzcd}
[column sep=large]
\Gal_{\Q_\ell}\arrow[dd,hook]\arrow[rd,"\smash{\phi_{T_{\Q_\ell},g}}",out=-45,in=180,pos=0.35]\arrow[r,"\phi_{T_{\fT,\ell}}"]&
\GL(\Xup(T_{\fT,\ell,\ol{\Q_\ell}}))\arrow[d,"\theta_b"]\arrow[r,"\theta_{a_\ell}"]&
\GL(\Xup(T_{0,\ol{\Q_\ell}}))\arrow[dd,"\sim"]\\
&\GL(\Xup(T_{\ol{\Q_\ell}}))\arrow[d,"\sim"]\arrow[ru,"\smash{\theta_{a}}",out=0,in=225]\\
\Gal_\Q\arrow[r,"\phi_{T,g}"]&
\GL(\Xup(T_{\ol\Q}))\arrow[r,"\theta_a"]&
\GL(\Xup(T_{0,\ol\Q})).
\end{tikzcd}
\]
It suffices to prove that the outer square commutes up to conjugation by $\W(\Psi_0)$.
By (\ref{eq:compatibility-2}--\ref{eq:compatibility-3}), the bottom squares commute on the nose, and the upper-right triangle certainly commutes up to conjugation by the Weyl group.
Finally, the upper-left triangle commutes by Lemma \ref{lem:phi-T-g-commutes}.
\end{proof}

\begin{thing}
\label{thing:frob-splitting}
We will denote by $F_\fT$ the splitting field of the map $\phi_\fT$ of Lemma \ref{lem:torus-info} ($\phi_\fT$ being a representative of conjugacy class $[\phi_\fT]$), i.e.\ the field such that $\phi_\fT$ factors through an injective map $\Gal(F_\fT|\Q)\into\Aut(\Psi_0)$.
By construction, $E\subseteq F_\fT$, and $F_\fT\Q_\ell$ is the splitting field of $T_{\fT,\ell}$ for each $\ell$, and $[F_\fT:\Q]\leq\#{\Aut(\Psi_0)}$.
\end{thing}

\subsection{Statement and proof}
\label{subsec:thm-A}

\begin{thm}
\label{thm:main}
Let $\{\rho_\ell\cln\Gamma\to G(\Q_\ell)\}_{\ell\in\cL}$ be a compatible collection of $G$-representations of an F-group, where $G$ is a connected reductive group over $\Q$ and $\cL$ is a set of Dirichlet density $1$.
Let $E|\Q$ be the splitting extension of $\phi_{G^\mr{ad}}$ as in \eqref{thing:main-thm-G}, and suppose for each conjugacy class $C$ of $\Gal(E|\Q)$, there exists $\ell_C\in\cL$ such that
\begin{enumerate}
\item
$\ell_C$ is unramified in $E|\Q$.
\item
the Frobenius conjugacy class of $\ell_C$ in $\Gal(E|\Q)$ is $C$.
\item
$G_{\Q_{\ell_C}}$ is quasisplit.
\item
$\ol{\Img(\rho_{\ell_C})}{}^\Zar=G_{\Q_{\ell_C}}$.
\end{enumerate}
Then the set
$
\{
\ell\in\cL
\cln
\ol{\Img(\rho_\ell)}{}^\Zar=G_{\Q_\ell}
\}
$
has Dirichlet density $1$.
\end{thm}

\begin{proof}
As before, we set $M_\ell\ceq\ol{\Img(\rho_\ell)}{}^\Zar$ for each $\ell\in\cL$.
Let us abuse notation by writing ``$M_\ell=G$'' to mean that $M_\ell=G_{\Q_\ell}$.
We will proceed in three steps:
\begin{enumerate}
\item[1.]
First, reduce to the case when $G=G^\mr{ad}$ and each $M_\ell$ is reductive and of maximal rank in $G$. (Then, in particular, the hypotheses of Lemma \ref{lem:torus-info} are satisfied.)
\item[2.]
Next, find infinitely many $\ell\in\cL$ completely split in $E|\Q$ such that $M_\ell=G$.
\item[3.]
Finally, pick an arbitrary conjugacy class $C$ of $\Gal(E|\Q)$ and show that for a set $\mc Z_C$ of Dirichlet density $0$, \tit{all} $\ell\in\cL\setminus\mc Z_C$ which are unramified in $E|\Q$ and whose Frobenius conjugacy class in $\Gal(E|\Q)$ is $C$ satisfy $M_\ell=G$.\phantom{\qedhere}
\end{enumerate}
\end{proof}

\begin{proof}[Step 1]
Consider the projections of the $\rho_\ell$ to $G^\mr{Ab}$ and $G^\mr{ad}$.
The resulting collections of, respectively, $G^\mr{Ab}$- and $G^\mr{ad}$-representations remain compatible, and the hypotheses (\tit a)--(\tit d) of the present theorem hold true for them.
Assume for a moment that the conclusion holds for these two compatible collections.
Then $M_\ell$, for $\ell$ in a density-1 set, surjects onto the codomain of the map $G_{\Q_\ell}\to G_{\Q_\ell}^\mr{ad}\times G_{\Q_\ell}^\mr{Ab}$.
On the other hand, this map has finite kernel, so $M_\ell$ is a finite-index subgroup of the connected group $G_{\Q_\ell}$ for such $\ell$, hence $M_\ell=G$.

Thus it suffices to prove the theorem in the cases when $G$ is a torus or an adjoint group.
The former is trivial because the rank of $M_\ell$ is independent of $\ell$ by \cite[\textsection3]{serre-letter} (see also \cite[Proposition 6.12]{larsen-pink}), so we henceforth assume $G=G^\mr{ad}$.

Now we pass to the $G$-semisimplifications; see \cite[\textsection\textsection3.2 and 4.1]{serre-reductibilite}.
Specifically, for each $\ell\in\cL$, let $P_\ell$ be a parabolic subgroup of $G_{\Q_\ell}$ minimal among those which contain $M_\ell$, and let $L_\ell$ be a Levi subgroup of $P_\ell$.
Finally, let $\rho_\ell^\ssimp$ be the composition $\Gamma\xrightarrow{\rho_\ell}P_\ell(\Q_\ell)\onto L_\ell(\Q_\ell)\into G(\Q_\ell)$.
Then $\rho_\ell^\ssimp(\gamma)_\ssimp=\rho_\ell(\gamma)_\ssimp$ for any $\gamma\in\Gamma$, so the $\rho_\ell^\ssimp$ again form a compatible collection of $G$-representations.
Moreover, $M_\ell=G$ if and only if $\ol{\Img(\rho_\ell^\ssimp)}{}^\Zar=G$, so by replacing each $\rho_\ell$ with $\rho_\ell^\ssimp$, we may and do assume each $M_\ell$ is a reductive subgroup of $G_{\Q_\ell}$.
And $M_\ell$ is of maximal rank in $G_{\Q_\ell}$ again by the $\ell$-independence of the rank.
This completes step 1.\phantom{\qedhere}
\end{proof}

\noindent
A preparatory remark on the rest of the proof: the embeddings $\iota_\ell$ are going to be chosen \tit{in the course of} the following arguments.
Doing so is not necessary for the proof but will allow us to use less notation---what we do use will be already quite a burden---at the cost of making what is actually happening somewhat less scrutable.

\begin{proof}[Step 2]
Fix an embedding $\iota_{\ell_{\{1\}}}$.
By \cite[Lemma 3.6]{larsen-pink} (stated as Lemma \ref{lem:exists-torus} above), for each conjugacy class $\xi$ of $\W(\Psi_0)$, there exists an unramified maximal torus $T_\xi$ of $M_{\ell_{\{1\}}}=G$ such that $[\phi_{T_\xi}](\Frob_{\ell_{\{1\}}})=\xi$.
Let $F_{\fT_\xi}$ be the field attached to $\fT_\xi\ceq\{(\ell_{\{1\}},T_\xi)\}$ as in \eqref{thing:frob-splitting}, and let $F$ be the composite of all $F_{\fT_\xi}$.
Then for all but finitely many $\ell\in\cL$ unramified in $F$ and such that $\ell$ and $\ell_{\{1\}}$ have the same Frobenius conjugacy class in $\Gal(F|\Q)$, if we pick $\iota_\ell$ such that the \tit{elements} defined by $\Frob_\ell$ and $\Frob_{\ell_{\{1\}}}$ are \tit{equal} in $\Gal(F|\Q)$, then $M_\ell$ possesses an unramified maximal torus corresponding to each conjugacy class $\xi$ of $\W(\Psi_0)$, namely $T_{\fT_\xi,\ell}$.
So $M_\ell=G$ for any such $\ell$ by Lemma \ref{lem:super-jordan}, and step 2 is done.\phantom{\qedhere}
\end{proof}

\begin{proof}[Step 3]
Fix a conjugacy class $C$ of $\Gal(E|\Q)$ and an embedding $\iota_{\ell_C}$,  and let $\Omega$ be the preimage of $\phi_G(\Frob_{\ell_C})$ in $\Aut(\Psi_0)$.
Fix $n\geq1$, and for each $i\in\{1,\dots,n\}$, each $\omega\in[\Omega]$, and each conjugacy class $\xi$ of $\W(\Psi_0)$, find a prime $\ell_{i,\omega,\xi}\in\cL$ completely split in $E|\Q$.
Assume all such $\ell_{i,\omega,\xi}$ are distinct from each other and from $\ell_C$ and (by step 2) satisfy $M_{\ell_{i,\omega,\xi}}=G$.
Fix an embedding for each such prime.
For each $i\in\{1,\dots,n\}$ and $\omega\in[\Omega]$, let $\fT_{i,\omega}\ceq\{(\ell_C,T_\omega)\}\cup\{(\ell_{i',\omega',\xi},T_{(i,\omega),(i',\omega',\xi)})\}_{i',\omega',\xi}$, where the tori are unramified and satisfy $[\phi_{T_\omega}](\Frob_{\ell_C})=\omega$ and
\begin{equation}
\label{eq:torus-info}
[\phi_{T_{(i,\omega),(i',\omega',\xi)}}](\Frob_{\ell_{i',\omega',\xi}})
=
\begin{cases}
\xi,&(i',\omega')=(i,\omega)\\
\{1\},&(i',\omega')\neq(i,\omega).
\end{cases}
\end{equation}
Finally, let $F_i$ be the composite of all $F_{\fT_{i,\omega}}$ as $\omega$ varies.
We claim the $F_i$ are linearly disjoint over $E$.
Granted this, we finish step 3 as follows.
For all but finitely many $\ell\in\cL$ unramified in $F_i$ and such that $\ell$ and $\ell_C$ have the same Frobenius conjugacy class in $\Gal(F_i|\Q)$, if we pick $\iota_\ell$ such that $\Frob_\ell$ and $\Frob_{\ell_{C}}$ define the same element of $\Gal(F_i|\Q)$, then $M_\ell=G$ as in step 2 by Lemma \ref{lem:super-jordan}.
Now by the claim and Chebotarev's density theorem, the set of all rational primes $\ell$ which have Frobenius conjugacy class $C$ in $\Gal(E|\Q)$ and are either ramified in some $F_i$ or such that $\ell$ and $\ell_C$ have a \tit{different} Frobenius conjugacy class in $\Gal(F_i|\Q)$ for \tit{each} $i$ is of Dirichlet density
\begin{equation}
\label{eq:density-bound}
\leq
\frac{\#C}{[E:\Q]}
\prod_{i=1}^n\left(1-\frac1{[F_i:E]}\right)
\leq
\frac{\#C}{[E:\Q]}
\left(1-\frac1{\#{\W(\Psi_0)\cdot\#[\Omega]}}\right)^n,
\end{equation}
and taking $n$ to infinity, we see that the set of $\ell$ which have Frobenius conjugacy class $C$ in $\Gal(E|\Q)$ but $M_\ell\neq G$ is of Dirichlet density $0$.
To prove \eqref{eq:density-bound}, we estimate the size of
\[
\Sigma
\ceq
\{
\sigma\in\Gal(F|\Q)
\cln
\sigma|_{F_i}\notin C_i\text{ for each }i\text{, but }\sigma|_E\in C
\},
\]
where $C_i$ is the Frobenius conjugacy class of $\ell_C$ in $\Gal(F_i|\Q)$.
Considering $\Gal(F|\Q)$ as a subgroup of $\prod_{i=1}^n\Gal(F_i|\Q)$ in the natural way, and letting $p_i\cln{\Gal(F_i|\Q)}\to\Gal(E|\Q)$ denote the restriction map, there is (by linear disjointness) a bijection
\[
\Sigma
\cong
\bigsqcup_{\tau\in C}
\left\{(\sigma_1,\dots,\sigma_n)
\in
\prod_{i=1}^n\Gal(F_i|\Q)
\cln
\sigma_i\in p_i^{-1}\{\tau\}\setminus C_i
\right\}.
\]
As $p_i^{-1}\{\tau\}$ intersects nontrivially with $C_i$, we have
$
\#\Sigma
\leq
\#C\prod_{i=1}^n([F_i:E]-1).
$
Combining this with the equality $[F:\Q]=[E:\Q]\prod_{i=1}^n[F_i:E]$ and Chebotarev's density theorem yields the desired upper bound on Dirichlet density.
 
Finally, we prove the claim of the previous paragraph.
In fact, we will show that the $F_{\fT_{i,\omega}}$ are linearly disjoint over $E$.
Fix $(i,\omega)$, and let $F^{(i,\omega)}$ be the composite of all $F_{\fT_{i_1,\omega_1}}$ for $(i_1,\omega_1)\neq(i,\omega)$; we must show that $F_{\fT_{i,\omega}}\cap F^{(i,\omega)}=E$.
Let $\phi_{\fT_{i,\omega}}\cln{\Gal(F_{\fT_{i,\omega}}|\Q)}\to\Aut(\Psi_0)$ be a representative of the conjugacy class $[\phi_{\fT_{i,\omega}}]$, and consider the diagram
\[
\begin{tikzcd}
1\arrow[r]&
\Gal(F_{\fT_{i,\omega}}|E)\arrow[d,hook]\arrow[r]&
\Gal(F_{\fT_{i,\omega}}|\Q)\arrow[d,hook,"\phi_{\fT_{i,\omega}}"']\arrow[r]&
\Gal(E|\Q)\arrow[d,hook,"\phi_G"']\arrow[r]&
1\\
1\arrow[r]&
\W(\Psi_0)\arrow[r]&
\Aut(\Psi_0)\arrow[r]&
\Out(\Psi_0)\arrow[r]&
1,
\end{tikzcd}
\]
whose rows are exact.
Fix any conjugacy class $\xi$ of $\W(\Psi_0)$, and let $\sigma\in\Gal(F_{\fT_{i,\omega}}|\Q)$ be the image of $\Frob_{\ell_{i,\omega,\xi}}$.
By construction, we have $\phi_{\fT_{i,\omega}}(\sigma)\in\xi$, but since $\ell_{i,\omega,\xi}$ splits completely in $E|\Q$, it is also the case that $\sigma\in\Gal(F_{\fT_{i,\omega}}|E)$.
Since $\xi$ may be chosen arbitrarily, the image of $\Gal(F_{\fT_{i,\omega}}|E)$ in $\W(\Psi_0)$ hits every conjugacy class, hence is all of $\W(\Psi_0)$, and in particular the conjugacy classes of $\Gal(F_{\fT_{i,\omega}}|E)$ are in bijection with those of $\W(\Psi_0)$.
On the other hand, $\ell_{i,\omega,\xi}$ splits completely in each $F_{\fT_{i_1,\omega_1}}|\Q$ by \eqref{eq:torus-info}, hence in $F^{(i,\omega)}|\Q$, so $\sigma$ becomes trivial in $\Gal(F_{\fT_{i,\omega}}\cap F^{(i,\omega)}\,|\,E)$, which proves (again since $\xi$ is arbitrary) that $F_{\fT_{i,\omega}}\cap F^{(i,\omega)}=E$.
\end{proof}

\section{Theorem \texorpdfstring{\ref{thm:B}}{B}}
\label{sec:thm-B-pf}

\subsection{Hilbert's irreducibility theorem for profinite groups}
\label{subsec:hilbert-irreducibility}

Recall that, if $\Pi$ is a profinite group, its Frattini subgroup $\Phi(\Pi)$ is defined to be the intersection of all maximal proper \tit{closed} subgroups of $\Pi$.
The following result is a slight refinement of \cite[Fact 3.3.1.1]{cadoret-kret}; its proof is essentially extracted from those of \cite[\textsection9.2, Proposition 2 and \textsection10.6, Theorem]{serre-mordell-weil}.

\begin{prop}
\label{prop:exist-specializations}
Let $X$ be a connected normal scheme of finite type over $\Z$ such that $\dim(X_\Q)\geq1$.
Then there are positive integers $N$ and $d$ with the following property:
For any continuous surjective homomorphism $\rho\cln\pi_1(X)\onto\Pi$ where $\Pi$ is a profinite group with open Frattini subgroup, there exist infinitely many closed points $x$ of $X_\Q$ such that
\begin{enumerate}
\item
$[\kappa_x:\Q]\leq d$.
\item
$\Img(\rho_x)=\Pi$.
\item
$x$ extends to an element of $X(\ol\Z[1/N])$.
\end{enumerate}
\end{prop}

\begin{proof}
By the surjectivity of $\pi_1(U)\to\pi_1(X)$ for open $U\subseteq X$ \cite[Lemma \href{https://stacks.math.columbia.edu/tag/0BQI}{\texttt{0BQI}}]{stacks}, we may assume $X$ is affine.
Pick a finite map $\pi\cln X_\Q\to\A^n_\Q$ of degree $d$, and let $N$ be a positive integer such that $\pi$ spreads out to a finite morphism $X_{\Z[1/N]}\to\smash{\A^n_{\Z[1/N]}}$.
We claim that these $N$ and $d$ work.

Fix $\rho\cln\pi_1(X)\onto\Pi$ as in the statement, and let $\overline\rho$ be the composition
\[
\pi_1(X_\Q)
\onto
\pi_1(X)\onto
\Pi\onto\Pi/\Phi(\Pi),
\]
the first arrow being surjective by \tit{loc.\ cit.}
If $x\in X_\Q$ satisfies $\Img(\overline\rho_x)=\Pi/\Phi(\Pi)$, then $\Img(\rho_x)\cdot\Phi(\Pi)=\Pi$, hence $\Img(\rho_x)=\Pi$ by the definition of $\Phi(\Pi)$.
In the following paragraph, we produce such $x$.

Since, by hypothesis, $\Pi/\Phi(\Pi)$ is finite, the map $\overline\rho$ corresponds to a finite \'etale cover $Y\onto X_\Q$ which is Galois with group $\Pi/\Phi(\Pi)$.
For each maximal proper subgroup $\Sigma$ of $\Pi/\Phi(\Pi)$, let $\pi_\Sigma$ denote the composition $Y/\Sigma\onto X_\Q\xrightarrow{\pi}\A^n_\Q$. Since $X$ is normal, $Y/\Sigma$ is irreducible \cite[Lemma \href{https://stacks.math.columbia.edu/tag/0BQL}{\texttt{0BQL}}]{stacks}, so by \cite[\textsection9.2, Proposition 1]{serre-mordell-weil}, there is a thin set $\Omega_\Sigma\subseteq\A^n(\Q)$ such that if $x_0\in\A^n(\Q)\setminus\Omega_\Sigma$, then the $\ol\Q$-points of the fiber $\pi_\Sigma^{-1}\{x_0\}$ consist of $\deg(\pi_\Sigma)$ points of $Y/\Sigma$ all conjugate over $\Q$.
Put $\Omega\ceq\bigcup_\Sigma\Omega_\Sigma$.
By \cite[\textsection9.6, Theorem]{serre-mordell-weil}, $\A^n(\Z[1/N])\setminus\Omega$ is an infinite set; let $x_0$ be an element and $x\in\pi^{-1}\{x_0\}$.
Then $x$ is a closed point of $X_\Q$ satisfying $[\kappa_x:\Q]\leq d$.
Since $x_0\notin\Omega$ and $\deg(\pi_\Sigma)>d$, the point $x$ does not lift to $(Y/\Sigma)(\kappa_x)$ for any $\Sigma$, so $\Img(\overline\rho_x)$ cannot be a subset of $\Sigma$, i.e.\ $\Img(\overline\rho_x)=\Pi/\Phi(\Pi)$.

Finally, since $X_{\Z[1/N]}\to\A^n_{\Z[1/N]}$ is finite and $\ol\Z[1/N]$ is the integral closure of $\Z[1/N]$ in $\ol\Q$, the point $x$ extends to an element of $X(\ol\Z[1/N])$, as desired.
\end{proof}

\noindent
Next, we describe many well-known examples of profinite groups having open Frattini subgroup (cf.\ \cite[\textsection10.6, Example 1]{serre-mordell-weil}).

\begin{lem}
\label{lem:open-frattini}
Let $\ell_1,\dots,\ell_n$ be primes, $\Pi$ a compact subgroup of $\GL_{d_1}(\Q_{\ell_1})\times\cdots\times\GL_{d_n}(\Q_{\ell_n})$.
Then the Frattini subgroup $\Phi(\Pi)$ is open in $\Pi$.
\end{lem}

\begin{proof}
We assume the $\ell_i$ are distinct and $\Pi\subseteq\prod_i\GL_{d_i}(\Z_{\ell_i})$.
Then $\Pi$ has an open subgroup of the form $\Pi_1\times\cdots\times\Pi_n$, such that for each $i$, the group $\Pi_i$ is pro-$\ell_i$.
Furthermore, each $\Pi_i$ is topologically finitely generated (see e.g.\ \cite[Theorem 5.2]{analytic-pro-p-groups}).
This verifies criterion (iv) of \cite[\textsection10.6, Proposition]{serre-mordell-weil}.
\end{proof}

\subsection{Statement and proof}
\label{subsec:thm-B}

\begin{thing}
Let $X$ be a connected normal scheme of finite type over $\Z$.
By Chebotarev's density theorem \cite[\textsection2.7, Theorem 7]{serre-chebotarev}, the set of Frobenius elements is dense in $\pi_1(X)$.
Together with the surjectivity of $\pi_1(X_\Q)\to\pi_1(X)$ \cite[Lemma \href{https://stacks.math.columbia.edu/tag/0BQI}{\texttt{0BQI}}]{stacks}, this implies that $(\pi_1(X_\Q),\cF_X)$ is an F-group in the sense of Definition \ref{defn:compatible}, where $\cF_X$ is the preimage of the set of Frobenius elements under the map $\pi_1(X_\Q)\to\pi_1(X)$.

We will say that a collection $\{\rho_\ell\cln\pi_1(X_{\Z[1/\ell]})\to G(\Q_\ell)\}_{\ell\in\cL}$ of $G$-representations, where $\cL$ is a set of primes, is compatible if the induced collection $\{\pi_1(X_\Q)\to G(\Q_\ell)\}_{\ell\in\cL}$ is compatible relative to $\cF_X$ in the sense of Definition \ref{defn:compatible}.
(This agrees with the notion of compatibility described in \eqref{thing:tate-stuff}.)
\end{thing}

\begin{thm}
\label{thm:main-2}
Let $X$ be a connected normal scheme of finite type over $\Z$ such that $\dim(X_\Q)\geq1$, and let $\{\rho_\ell\cln\pi_1(X_{\Z[1/\ell]})\to G(\Q_\ell)\}_{\ell\in\cL}$ be a compatible collection of $G$-representations, where $G$ is a connected reductive group over $\Q$ and $\cL$ is a set of Dirichlet density $1$.
Assume that $\ol{\Img(\rho_\ell)}{}^\Zar=G_{\Q_\ell}$ for each $\ell\in\cL$.
Let $E|\Q$ denote the splitting extension of $\phi_{G^\mr{ad}}$ as in \eqref{thing:main-thm-G}.

Let $x$ be a closed point of $X_\Q$.
Assume there exists a positive integer $N$ such that $x$ extends to an element of $X(\ol\Z[1/N])$ and a finite set $\cR\subseteq\cL$ with the following property:
For each conjugacy class $C$ of $\Gal(E|\Q)$ there exists $\ell\in\cR$ such that
\begin{enumerate}
\item
$\ell$ is unramified in $E|\Q$.
\item
the Frobenius conjugacy class of $\ell$ in $\Gal(E|\Q)$ is $C$.
\item
$G_{\Q_\ell}$ is quasisplit.
\item
$\ell\nmid N$.
\item
$\ol{\Img(\rho_{\ell,x})}{}^\Zar=G_{\Q_\ell}$.
\end{enumerate}
Then the set
$
\{
\ell\in\cL
\cln
\ol{\Img(\rho_{\ell,x})}{}^\Zar
=
G_{\Q_\ell}
\}
$
has Dirichlet density $1$.

Moreover, there exists a positive integer $d$ and infinitely many $x$ as in the previous paragraph satisfying $[\kappa_x:\Q]\leq d$.
\end{thm}

\begin{proof}
If $(x,N,\cR)$ is as in the statement of the present theorem, the collection $\{\rho_{\ell,x}\cln{\Gal_{\kappa_x}}\to G(\Q_\ell)\}_{\ell\in\cL}$ is compatible relative to $\cF_{\Spec(\cO_{\kappa_x}[1/N])}$, hence satisfies the hypotheses of Theorem \ref{thm:main}, which gives the statement about Dirichlet density.

We now explain how to find such $(x,N,\cR)$.
Let $N$ be as Proposition \ref{prop:exist-specializations}, and let $\cR\subseteq\cL$ be any set such that for each $C$ as above there exists $\ell\in\cR$ satisfying (\tit a)--(\tit d).
Considering the product map
\[
\rho_\cR
\cln
\pi_1(X_\Q)
\to
\prod_{\ell\in\cR}G(\Q_\ell),
\]
use Proposition \ref{prop:exist-specializations} and Lemma \ref{lem:open-frattini} to find a closed point $x$ of $X_\Q$ such that $\Img(\rho_{\cR,x})=\Img(\rho_\cR)$ and $x$ extends to an element of $X(\ol\Z[1/N])$.
Then $(x,N,\cR)$ has all the desired properties.
\end{proof}

\bibliographystyle{amsalpha}
\bibliography{tozd.bib}

\end{document}